\documentclass[a4paper,10pt]{article}
\usepackage[utf8]{inputenc}
\usepackage{amsmath,amssymb,amsthm,enumerate}
\usepackage[mathscr]{eucal}

\usepackage{tikz}
\usetikzlibrary{matrix,arrows}

\theoremstyle{plain}
\newtheorem{thm}{Theorem}
\newtheorem{lemma}{Lemma}
\newtheorem{rmk}{Remark}

\theoremstyle{definition}
\newtheorem{defn}{Definition}
\newtheorem{prop}{Proposition}
\newtheorem{cor}{Corollary}

\newcommand{\ba}{\boldsymbol{\mathscr{A}}}
\newcommand{\bg}{\boldsymbol{\mathscr{G}}}
\newcommand{\llg}{\langle \langle}
\newcommand{\rrg}{\rangle\rangle}
\newcommand{\lgl}{\langle}
\newcommand{\rgl}{\rangle}

\title{Explicit Construction of Local Langlangds Correspondence of $GL(2,F)$ Using Theta Correspondence}
\author{Ran Cui}

\begin{document}

\maketitle

\begin{abstract}
The Langlands correspondence of $GL(2,F)$ over a non-Archimedean local field $F$ has been well studied in \cite{BH} and \cite{Bump}. One important step in constructing this correspondence is to construct a correspondence between the irreducible imprimitive 2-dimensional Deligne representations of the Weil group of $F$ and the irreducible cuspidal smooth representations of $GL(2,F)$. In this paper, we are going to describe explicitly how this construction is done. We will then provide new proofs to several properties of the correspondence. Including that the correspondence does not depend on the additive character $\psi$ introduced by the Weil representation.
\end{abstract}

\section{Introduction}

Throughout this paper, $F$ will denote a non-Archimedean local field with characteristics 0. The local Langlands conjecture says there exists a unique correspondence between the set 
\[
 \bg_2(F) :=\{ \text{2-dim, semi-simple Deligne representations of the Weil group } \mathcal{W}_F\}/ \sim
\]
and the set 
\[
 \ba_2(F) := \{ \text{irreducible smooth representations of }GL(2,F)\}/\sim 
\]
with certain properties:

\begin{defn}\label{locallanglands}[Local Langlands Conjecture] There is a unique bijection 
\[
 \boldsymbol{\pi} : \bg_2(F)\rightarrow \ba_2(F)
\]
such that 
\begin{equation}\label{Lcondition}
 L(\chi({\boldsymbol{\pi}}(\rho)),s) = L(\chi\otimes \rho, s)
\end{equation}
\begin{equation}\label{epsiloncondition}
 \epsilon(\chi({\boldsymbol{\pi}}(\rho)),s,\psi) = \epsilon(\chi\otimes \rho,s,\psi)
\end{equation}
for all $\rho\in \bg_2(F)$ and all characters $\chi$ of $F^{\times}$ and all $\psi\in \widehat{F}$, $\psi\neq 1$.

Here $L(\pi,s)$ is the L-function and $\epsilon(\pi,s,\psi)$ is the local-constant \cite{BH}. And
\[
 \chi\pi = \pi\otimes (\chi\circ \det)
\]

\end{defn}

We have the following decomposition:
 \[
  \bg_2^0(F) = \{\text{irreducible, 2-dim, semi-simple Deligne representation of }\mathcal{W}_F\}
 \]
 \[
  \bg_2^1(F) = \{\text{reducible, 2-dim, semi-simple Deligne representations of }\mathcal{W}_F\}
 \]
 
 \[
  \bg_2(F) = \bg_2^0(F)\cup \bg_2^1(F)
 \] 
 \[
  \ba_2^0(F) = \{ \text{irreducible cuspidal smooth representations of }GL(2,F)\}
 \]
 \[
  \ba_2^1(F) = \{ \text{irreducible non-cuspidal smooth representations of }GL(2,F)\}
 \]
 
 \[
  \ba_2(F) = \ba_2^0(F)\cup \ba_2^1(F)
 \]
 \begin{prop}\cite[P213]{BH}
  Any map $\boldsymbol{\pi}$ satisfying equation (\ref{Lcondition}) and (\ref{epsiloncondition}) must take $\bg_2^1(F)$ to $\ba_2^1(F)$ and $\bg_2^0(F)$ to $\ba_2^0(F)$.

 \end{prop}
 
 The map 
 \[
  \boldsymbol{\pi^1}: \bg_2^1(F)\rightarrow \ba_2^1(F)
 \]
is clearly presented in \cite[P213]{BH}.

The heart of the matter is therefore to construct map:
\[
 \boldsymbol{\pi}:\bg_2^0(F)\rightarrow \ba_2^0(F)
\]

There is more than one way to construct this map. Here, we will be using the theta correspondence. The outline is described in the following diagram:
\begin{center}
\begin{tikzpicture}[description/.style={fill=white,inner sep=2pt}]
\matrix (m) [matrix of math nodes, row sep=4em,
column sep=4em, text height=1.5ex, text depth=0.25ex]
{ \rho_{\Theta} & (E/F,\Theta) & \theta \\
\pi_{\Theta}& \pi_{\Theta,\psi} & \pi_{\theta,\psi} \\ };
\path[<->,font=\scriptsize]
(m-2-1) edge node[description]{local Langlands correspondence} (m-1-1);
\path[->,font=\scriptsize]
(m-1-1) edge node[auto] {Theorem\ref{admpairsparamG}} (m-1-2)
(m-1-2) edge node[auto] {$\Theta|_{E^1}$} (m-1-3)
(m-1-3) edge node[description] {theta correspondence} (m-2-3)
(m-2-3) edge node[auto] {extension} (m-2-2)
(m-2-2) edge node[auto] {induction} (m-2-1);
\end{tikzpicture}
\end{center}

The process goes clockwise beginning at $\rho_{\Theta}\in \bg_2^0(F)$. 
\begin{enumerate}
 \item $\bg_2^0(F)\ni \rho_{\Theta}\mapsto (E/F,\Theta)\in \mathbb{P}(F)$, this map is described in Theorem~\ref{admpairsparamG}.
 \item $\mathbb{P}(F)\ni (E/F,\Theta)\mapsto \theta:= \Theta|_{E^1},$ here $E^1\cong SO(2,F)$.
 \item $\widehat{SO}(2,F)\ni \theta\mapsto \pi_{\theta,\psi}\in \widehat{SL}(2,F)$. This map is produced by (some modification of) the theta correspondence.
 \item $\widehat{SL}(2,F)\ni \pi_{\theta,\psi} \mapsto \pi_{\Theta,\psi}\in \widehat{GL}(2,F)$. The representation $\pi_{\Theta,\psi}$ is obtained by extending $\pi_{\theta,\psi}$ and then inducing the extended representation.
\end{enumerate}

The core of the construction is to use the theta correspondence for the dual pair $(O(2,F),Sp(2,F))$ in $Sp(4,F)$. This allows us to get a representation $\pi_{\theta,\psi}$ of $Sp(2,F)\cong SL(2,F)$ from $\theta$. Note that the dual pair correspondence depends on $\psi$.

Since the Langlands correspondence should not involve $\psi$, naturally we would like to prove the representation $\pi_{\Theta}$ is independent of $\psi$. Bushnell and Henniart \cite{BH} used the local constant to prove this fact. Here we derive this fact more directly from properties of the Weil representation. In addition, the irreducibility of $\pi_{\Theta}$ is also naturally built in.

\section{Admissible Pairs}

Most of the content in this section is from \cite{BH}. 

Consider a pair $(E/F,\chi)$, where $E/F$ is a tamely ramified quadratic field extension and $\chi$ is a character of $E^{\times}$

\begin{defn}
 The pair $(E/F,\chi)$ is called \emph{admissible} if:
 \begin{enumerate}
  \item $\chi$ does not factor through the norm map $N_{E/F}: E^{\times}\rightarrow F^{\times}$ and,
  \item if $\chi|U_E^1$ does factor through $N_{E/F}$, then $E/F$ is unramified.
 \end{enumerate}
 Set $U_E^1 = 1+\mathfrak{p}$ where $\mathfrak{p}$ is the maximal ideal of the discrete valuation ring in $E$.
\end{defn}

\begin{defn}\cite[P127]{BH}
 Admissible pairs $(E/F,\chi)$, $(E^{\prime}/F,\chi^{\prime})$ are said to be \emph{$F$-isomorphic} if there is an $F$-isomorphism $j: E\rightarrow E^{\prime}$ such that $\chi = \chi^{\prime}\circ j$. In the case $E = E^{\prime}$, this amounts to $\chi^{\prime} = \chi^{\sigma}$, $\sigma\in Gal(E/F)$.
\end{defn}
Let $\mathbb{P}(F)$ be the $F$-isomorphism classes of admissible pairs $(E/F,\chi)$

\begin{thm}\cite{BH}\label{admpairsparamG}
 There is a bijection $\mathbb{P}(F)\leftrightarrow \bg_2^0(F)$.
\end{thm}

Here's how the map is constructed.
 
 If $(E/F,\chi)$ be an admissible pair, then $\chi$ can be viewed as a character of the Weil group $\mathcal{W}_E$ by composing with the Artin map: $\boldsymbol{a}_E:\mathcal{W}(E)\rightarrow E^{\times}$. We know that $\mathcal{W}(E)$ is an index two subgroup in $\mathcal{W}(F)$, we can induce $\chi$:
 \[
  \rho_{\chi} = \text{Ind}_{\mathcal{W}(E)}^{\mathcal{W}(F)}\chi
 \]
 This induced representation is irreducible by Clifford theory, since $\chi$ does not factor through $N_{E/F}$ is equivalent to the fact that $\chi\neq \chi^{\sigma}$, and the Artin map is $\text{Gal}(E/F)$-equivalent. 

The proof of the bijectivity of this map is given in \cite{BH}. 

 \begin{rmk}
  The admissible pairs $\mathbb{P}(F)$ also parametrize $\ba_2^0(F)$, and it can therefore give a bijection between $\bg_2^0(F)$ and $\ba_2^0(F)$, but it is well known that this map does not satisfy equation (\ref{Lcondition}) and (\ref{epsiloncondition}). There are various ways to define the correct correspondence. In the next section we will introduce the basics of the theory of theta correspondence.
 \end{rmk}

\section{The Theta Correspondence}

We are going to explain the theory of theta correspondence in the generality that suits the purpose of this paper. First of all, we define a basic ingredient of the theta correspondence: a dual reductive pair $(G, G^{\prime})$ in $Sp(W)$. Then we will present the Schr\"odinger Model of Weil representation of the Metaplectic group $Mp(W)$. The restriction of the Weil representation to the cross product of the lifts $\widetilde{G}\times\widetilde{G^{\prime}}$ gives the theta correspondence between representations of $\widetilde{G}$ and $\widetilde{G^{\prime}}$. We apply this to $\widetilde{O}(E)\times \widetilde{Sp}(2,F)$.

Most of the material and detailed proofs in this section can be found in \cite{Kudla}, \cite{Prasad} and \cite{Howe}.

\subsection{Dual Reductive Pair}

\begin{defn}
 A \emph{reductive dual pair} is a pair of subgroups $(G,G^{\prime})$ in the symplectic group $Sp(W)$ such that: 
 \begin{enumerate}
  \item $G$ is the centralizer of $G^{\prime}$ in $Sp(W)$, and $G^{\prime}$ is the centralizer of $G$ in $Sp(W)$.
  \item the actions of $G$ and $G^{\prime}$ are completely reducible on $W$, i.e., any $G$-invariant subspace of $W$ has a $G$-invariant complement; similarly for $G^{\prime}$.
 \end{enumerate}
\end{defn}

What we will be using is a specific kind of dual pair. Let $V$ and $W$ be finite dimensional vector spaces over $F$; $V$ is equipped with a symmetric bilinear form $(,)$; and $W$ is equipped with a skew-symmetric bilinear form $\lgl, \rgl$. Necessarily $\dim(W)$ has to be even.

Let $\mathbb{W} = V\otimes W$ with symplectic form $\llg, \rrg$ = $ (,)\otimes \lgl,\rgl$.
The pair $(O(V),Sp(W))$ is a dual reductive pair in $Sp(\mathbb{W})$.

\subsection{Weil Representation and the Schr\"odinger Model}

Throughout this section $W$ is a symplectic vector space with non-degenerate symplectic bilinear form $\lgl, \rgl$ represented by $\begin{pmatrix}0&I \\ -I&0\end{pmatrix}$. Fix $X$ and $Y$ two transversal Lagrangian subspaces of $W$.

\begin{defn}
 The \emph{Heisenberg group} $H(W)$ is a non-trivial extension of $W$ of $F$. It is defined to be the group of pairs 
 \[
  \{ (w,t)| w\in W, t\in F\} 
 \]
with multiplication law:
\[
 (w_1,t_1)\cdot (w_2,t_2) = (w_1+w_2, t_1+t_2+\frac{1}{2}\lgl w_1,w_2\rgl)
\]
\end{defn}
Thus $H(W)$ fits in the exact sequence 
\[
 0\rightarrow Z \rightarrow H(W) \rightarrow W \rightarrow 0
\]
where $Z = \{ (0,t)| t\in F\} \cong F$ is the center of $H(W)$.

The irreducible unitary representations of $H(W)$ are classified by their central characters. 

\begin{thm}[Stone-von Neumann] \label{stonevonneumann}
 The Heisenberg group $H(W)$ has a unique (up to unitary equivalence) irreducible unitary representation $\rho_{\psi}$ with central character $\psi$.
\end{thm}

The proof is given in \cite{MVW}.

These representations can be realized on the Schwartz space $\mathcal{S}=\mathcal{S}(X)$, the space of locally constant compactly supported complex valued functions on $X$. We will now give a construction of $(\rho_{\psi},\mathcal{S})$. Here we follow \cite{Kudla} very closely.

Let $H(Y)\cong Y\oplus Z$ be the Heisenberg group associated to $Y$ in $W$. Thus $H(Y)$ is a maximal abelian subgroup of $H(W)$. The character $\psi$ of $Z$ has a unique extension $\psi_Y$ to $H(Y)$, given by $\psi_Y(y,t) = \psi(t)$. Let 
\[
 S_Y = \text{Ind}_{H(Y)}^{H(W)} \psi_Y
\]
be the representation of $H(W)$ obtained from the character $\psi_Y$ by smooth induction. By definition:
\[
 S_Y = \{ f: H(W)\rightarrow \mathbb{C}| f(h_1h) = \psi_Y(h_1)\cdot f(h), \forall h_1\in H(Y), h\in H(W)
 \]
 \[
 \text{ and there is an open subgroup }L\subset W \text{ such that } f(h(u,0)) = f(h)\text{ for all }u\in L\}.
\]
Then $H(W)$ acts on $S_Y$ by right translation, in particular, $Z$ acts by the following:
\[
 (0,t)\cdot f(h) = f(h\cdot (0,t)) = f((0,t)\cdot h) = \psi_Y((0,t))f(h) = \psi(t)f(h) 
\]
Thus the central character is $\psi$. Denote this action of $H(W)$ on $S_Y$ as $\rho_{\psi}$.

Since $W = X+Y$ is a complete polarization of $W$, we have the following isomorphism:
\[
 S_Y\cong \mathcal{S}
\]
given by $f\mapsto \varphi$, $\varphi(x) = f((x,0))$.

\begin{lemma}
 The resulting action of $H(W)$ on $\mathcal{S}$ is given by:
 \[
   \rho_{\psi}((x+y,t))\cdot \varphi(x_0) = \psi(t+\lgl x_0, y\rgl + \frac{1}{2}\lgl x,y \rgl ) \cdot \varphi (x_0+x)
 \]

\end{lemma}

Now, we will explain how the representations of $H(W)$ give rise to a projective representation of $Sp(W) = \{ g\in GL(W)| \lgl w_1g, w_2g\rgl = \lgl w_1,w_2\rgl \}$.

The group $Sp(W)$ acts on $H(W)$ naturally:
\[
 g\cdot (w,t) = ( wg, t), \ \ g\in Sp(W)
\]
therefore acts on the representation $\rho_{\psi}$:
\[
 g\cdot \rho_{\psi} = \rho_{\psi}^g, \ \ \text{where } \rho_{\psi}^g((w,t)) = \rho_{\psi}(g\cdot (w,t)) = \rho_{\psi}(wg, t)
\]
It is easy to check that $\rho_{\psi}^g$ still has central character $\psi$. By Theorem~\ref{stonevonneumann}, we know $\rho_{\psi}\cong \rho_{\psi}^g$, therefore $g$ gives rise to $T(g): \mathcal{S}\rightarrow  \mathcal{S}$ such that the following diagram commutes:

\begin{center}
\begin{tikzpicture}[descr/.style={fill=white,inner sep=2.5pt}]
\matrix (m) [matrix of math nodes, row sep=3em,
column sep=3em]
{ \mathcal{S} & \mathcal{S} \\
\mathcal{S}& \mathcal{S} \\ };
\path[->,font=\scriptsize]
(m-1-1) edge node[auto] {$ T(g) $} (m-1-2)
edge node[auto,swap] {$ \rho_{\psi} $} (m-2-1)

(m-1-2) edge node[auto] {$\rho_{\psi}^g$} (m-2-2)

(m-2-1) edge node[auto] {$ T(g) $} (m-2-2);
\end{tikzpicture}
\end{center}
The intertwining operator $T(g)$ is unique up to scalar in $\mathbb{C}^{\times}$, so the map 
\[
 g\mapsto T(g)
\]
defines a projective representation of $Sp(W)$ on $\mathcal{S}$, i.e., a homomorphism 
\[
 Sp(W)\rightarrow GL(\mathcal{S})/\mathbb{C}^{\times}
\]
We also present a realization of this projective representation based on the previous realization of $\rho_{\psi}$ on $\mathcal{S}$.

Consider $W$ with the complete polarization we have fixed in the beginning of this section: $W = X+Y$, after choosing basis for $X$ and $Y$, we can write $g\in Sp(W)$ as 
\[
 g = \begin{bmatrix} A & B \\ C& D \end{bmatrix}
\]
where $A\in \text{End}(A)$, $B\in \text{Hom}(X,Y)$, $C\in \text{Hom}(Y,X)$ and $D\in \text{End}(Y)$.

\begin{defn}\cite{Kudla}
 The \emph{standard Segal-Shale-Weil representation} $\omega_{\psi}$ of $Sp(W)$ is realized on $\mathcal{S}$: 
 \[
  \left(\omega_{\psi}(g)\cdot \varphi\right) (x) = \int_{\ker C \backslash Y} \psi\left( \frac{1}{2}\lgl xA+yC, xB\rgl + \frac{1}{2}\lgl yC,xB+yD\rgl \right) \varphi(xA+yC)\ d\mu_g(y)
 \]
  $d\mu_g(y)$ is the unique Haar measure on $(Yg^{-1}\cap Y)\backslash Y$, hence on $(\ker C)\backslash Y$, such that $\omega_{\psi}(g)$ preserves the $L^2$ norm on $\mathcal{S}$.

\end{defn}
\begin{rmk}
 Note that this construction depends on the character $\psi$ of $F$ and the symplectic basis we chose. For detailed calculations, see \cite{RR} and \cite{Kudla}.
\end{rmk}

This realization determines a cocycle, i.e., a map $c: Sp(W)\times Sp(W)\rightarrow \mathbb{C}^*$:
\[
 \omega_{\psi}(g_1g_2) = c(g_1,g_2)\omega_{\psi}(g_1)\omega_{\psi}(g_2) 
\]
satisfying the cocycle condition:
\[
 c(g_1,g_2)c(g_1g_2,g_3) = c(g_1,g_2g_3)c(g_2,g_3)
\]
\begin{thm}\label{cocyclecanbenormalized}
The cocycle $c$ can be explicitly calculated and normalized such that it is $\pm 1$ valued. More specifically, we can find normalizing constant $m(g)$ such that the cocycle $\tilde{c}$ associated to the projective representation $\omega_{\psi}^{\prime}:=m(g)\omega_{\psi}(g)$ is $\pm 1$ valued.
\end{thm}

\begin{rmk}
 In order to avoid putting heavy notations in this section, we will present explicit formulas for $m(g)$ and the normalized cocycle in the Appendix.
\end{rmk}

\begin{defn}[Weil representation]
 The projective representation gives rise to an ordinary representation of the two-fold covering space of $Sp(W)$, it is called the \emph{Weil representation}
\end{defn}

\begin{defn}[Schr\"odinger Model]
\[
 Mp(W) = \{ (g,\epsilon)| g\in Sp(W),  \epsilon = \pm 1\}
\]
with group multiplication 
\[
 (g_1,\epsilon)\cdot (g_2,\delta) = (g_1 g_2 , \epsilon \delta \tilde{c}(g_1,g_2))
\]
is the Metaplectic group and the representation of $Mp(W)$ on $\mathcal{S}$ 
\[
 \tilde{\omega}_{\psi}(g,\epsilon) = \epsilon\ \omega_{\psi}^{\prime}
\]
is called the \emph{Schr\"odinger model of the Weil representation}.
\end{defn}

One property of the Weil representation we obtain from construction is the following:
\begin{cor}\cite[1.11]{Howe}\label{weilreplemma}
 \[
  \tilde{\omega}_{\psi_t} \cong \tilde{\omega}_{\psi_{ts^2}}
 \]
for $t, s\in F^{\times}$
\end{cor}

\subsection{Theta Correspondence}

The foundation of the theory is due to Howe.

We will start with a brief introduction to the general theory, then specialize to the case we are interested in. Here we follow \cite{Kudla} closely.

Let $(G,G^{\prime})$ be a dual pair in $Sp(W)$, $\widetilde{G}$ and $\widetilde{G^{\prime}}$ be the inverse image of $G$ and $G^{\prime}$ in $Mp(W)$.

\begin{lemma}\cite[Lemma 2.5]{MVW}
 If two elements in $Sp(W)$ commute then their arbitrary lifts in $Mp(W)$ also commute.
\end{lemma}

Therefore, we have a natural homomorphism:
\[
 j: \widetilde{G}\times \widetilde{G^{\prime}}\rightarrow Mp(W)
\]
hence we can consider the pull back of the Weil representation to $\widetilde{G}\times \widetilde{G^{\prime}}$, namely $\tilde{\omega}_{\psi}\circ j$, by abuse of notation we denote this representation $\tilde{\omega}_{\psi}$.

The basic idea of theta correspondence is that the Weil representation of $Mp(W)$ is very ``small'', its restriction to $\widetilde{G}\times \widetilde{G^{\prime}}$ should decompose into irreducibles in a reasonable way.

Suppose $\pi$ is an irreducible admissible representation of $\widetilde{G}$. Let $S(\pi)$ be the maximal quotient of $\mathcal{S}$ on which $\widetilde{G}$ acts as a multiple of $\pi$. By \cite[Chapter 2]{MVW}, there is a smooth representation $\Theta_\psi(\pi)$ of $\widetilde{G^{\prime}}$ such that 
\[
 S(\pi)\cong \pi\otimes \Theta_{\psi}(\pi)
\]
and $\Theta_{\psi}(\pi)$ is unique up to isomorphism.

\begin{thm}[Howe Duality Principle]
 For any irreducible admissible representation $\pi$ of $\widetilde{G}$ 
 \begin{enumerate}
  \item Either $\Theta_{\psi}(\pi) =0$ or $\Theta_{\psi}(\pi)$ is an admissible representation of $\widetilde{G^{\prime}}$ of finite length.
  \item If $\Theta_{\psi}(\pi)\neq 0$, there is a unique $\widetilde{G^{\prime}}$ invariant submodule $\Theta_{\psi}^0(\pi)$ such that the quotient 
  \[
   \theta_{\psi}(\pi):=\Theta_{\psi}(\pi)/\Theta_{\psi}^0(\pi)
  \]
  is irreducible.
  \item If $\theta_{\psi}(\pi_1)$ and $\theta_{\psi}(\pi_2)$ are nonzero and isomorphic, then $\pi_1\cong \pi_2$.

 \end{enumerate}

\end{thm}

\begin{defn}[Theta Correspondence]
 Let 
 \[
  Howe_{\psi}(\widetilde{G},\widetilde{G^{\prime}}) = \{ \pi \in \text{Irr}(\widetilde{G})|\theta_{\psi}(\pi) \neq 0\}
 \]
 be the set of (up to isomorphism) irreducible admissible representations of $\widetilde{G}$.
The map:
\[
 \pi\mapsto \theta_{\psi}(\pi)
\]
defines a bijection:
\[
 Howe_{\psi}(\widetilde{G},\widetilde{G^{\prime}}) \xrightarrow{\sim} Howe_{\psi}(\widetilde{G^{\prime}},\widetilde{G})
\]
this bijection is referred to as the \emph{local theta correspondence}.
\end{defn}

We now specialize to the case we are interested in: 

Let $E = F(\delta)$ be a tamely ramified quadratic field extension of $F$, $\delta = \sqrt{\Delta},\ \  \Delta\in F^{\times}/(F^{\times})^2$. The field $E$ can be viewed as a 2-dimensional vector space over $F$ with symmetric bilinear form induced by the norm map:
\[
 N(\alpha_1+\alpha_2\delta) = \alpha_1^2-\Delta\alpha_2^2
\]
with the basis: $\{ 1, \delta\}$, the form is represented by 
$
 \begin{pmatrix} 1&0 \\ 0& -\Delta \end{pmatrix}
$

Let $W$ be a 2-dimensional symplectic vector space over $F$ with basis $w_1,w_2$ and skew-symmetric form (after choosing basis) represented by
$ \begin{pmatrix}
  0&1 \\ -1&0
 \end{pmatrix}$

Consider $\mathbb{W} = W\otimes V$ with basis $\{  x_1:= w_1\otimes 1, x_2:= w_1\otimes -\frac{1}{\Delta}\delta, y_1:= w_2\otimes 1, y_2:= w_2\otimes \delta\} $, the skew-symmetric bilinear form is represented by 
\[
 \begin{pmatrix}
  0&0&1&0 \\ 0&0&0&1 \\ -1&0&0&0 \\ 0&-1&0&0
 \end{pmatrix}
\]

 Let $E^1 = \{ \alpha\in E| N(\alpha) = 1\}$. $E^1$ acts on $E$ by multiplication on the left, it's easy to check that $E^1\cong SO(E)$ by map: 
\[
 \alpha_1+\alpha_2\delta\mapsto \begin{bmatrix} \alpha_1&\alpha_2 \\ \alpha_2\Delta & \alpha_1 \end{bmatrix}
\]
and note that $Sp(2)\cong SL(2,F)$. By the theta correspondence and the Schr\"odinger model of Weil representation, we have the following result:

\begin{thm}
 Suppose $\theta$ is a regular character of $E^1$. The corresponding irreducible representation $\pi_{\theta,\psi}$ of $SL(2,F)$ on the space $\mathcal{S}(E)_{\theta} = \{ \varphi \in \mathcal{S}(E)| \varphi(\alpha\cdot \mu) = \theta(\mu)\varphi(\alpha),\ \forall \mu\in E^1,\ \forall \alpha\in E\}$ is described as follows:
 \[
  \pi_{\theta,\psi}\begin{bmatrix} a&0 \\ 0&a^{-1}\end{bmatrix}\cdot \varphi(\alpha) = (a,\Delta)\cdot |a|\cdot \varphi(a \alpha)
 \]
 \[
  \pi_{\theta,\psi}\begin{bmatrix} 1&b \\ 0&1\end{bmatrix}\cdot \varphi(\alpha) = \psi(b\cdot N(\alpha))\cdot \varphi(\alpha)
 \]
\[
 \pi_{\theta,\psi}\begin{bmatrix} 0&1 \\ -1&0 \end{bmatrix} \cdot \varphi(\alpha) = \gamma(\Delta,\psi)\hat{\varphi}(\alpha)
\]
where $\hat{\varphi} = \int_E\psi_E(z\alpha^{\sigma})\varphi(z)\ dz$ is the $\sigma$-twisted Fourier transform. See the Appendix for definitions of $\gamma$ and $(a,\Delta)$.
\end{thm}
\begin{proof}
 It is well known that the Metaplectic cover $Mp(4)$ splits over $SO(2)$. It also splits over $SL(2,F)$ because the orthogonal vector space $V$ has even dimension. Let's first describe the explicit embedding of $SO(2)\cong E^1$ in $\widetilde{O}(E)$ and $SL(2,F)$ in $\widetilde{SL}(2,F)$.
 Define:
 \[
 \iota: O(E)\times Sp(2)\rightarrow Sp(4)
\]
\[
 \begin{bmatrix} t&u\\r&s\end{bmatrix}\times \begin{bmatrix} a&b\\c&d \end{bmatrix} \mapsto \begin{bmatrix} at & -\Delta au&bt&bu \\ -\frac{ar}{\Delta}&as&-\frac{br}{\Delta}&-\frac{bs}{\Delta}\\ct&-\Delta cu&dt&du\\ cr&-\Delta cs&dr&ds\end{bmatrix}
\]
  For the convenience of future reference, we also write down the embeddings:
 \[
 \iota: O(E)\times 1\rightarrow Sp(4)
\]
\[
 \begin{bmatrix} t&u\\r&s\end{bmatrix}\times \begin{bmatrix} 1&0\\0&1 \end{bmatrix} \mapsto \begin{bmatrix} t & -\Delta u&0&0 \\ -\frac{r}{\Delta}&s&0&0\\0&0&t&u\\ 0&0&r&s\end{bmatrix}
\]

\[
 \iota: 1\times Sp(2)\rightarrow Sp(4)
\]
\[
 \begin{bmatrix} 1&0\\0&1\end{bmatrix}\times \begin{bmatrix} a&b\\c&d \end{bmatrix} \mapsto \begin{bmatrix} a & 0&b&0 \\ 0&a&0&-\frac{b}{\Delta}\\c&0&d&0\\ 0&-\Delta c&0&d\end{bmatrix}
\]
The embedding $E^1\hookrightarrow \widetilde{O}(E)$ is:
\[
 \alpha_1+\alpha_2\delta\mapsto \begin{bmatrix} \alpha_1&\alpha_2 \\ \alpha_2\Delta & \alpha_1 \end{bmatrix}\mapsto \left(\iota(\begin{bmatrix} \alpha_1&\alpha_2 \\ \alpha_2\Delta & \alpha_1 \end{bmatrix}), 1\right)
\]
It is a straightforward calculation that this is a group homomorphism given the formulas in \cite{RR}.

To obtain the embedding $SL(2,F)\hookrightarrow \widetilde{SL}(2,F)$, we look at \cite{KudlaSplitting}: 
\[
 c(\iota(g_1),\iota(g_2)) = \beta_V(g_1g_2)\beta_V(g_1)^{-1}\beta_V(g_2)^{-1}
\]
where 
\[
 \beta_V(g) = \gamma(x(g),\frac{1}{2}\psi))^{-m}(x(g),\det(V))\gamma(\frac{1}{2}\psi\circ V)^{-j}
\]
$m=\dim V$, $g$ is in the j-th cell $P\tau_j P$. 

The embedding $s: SL(2,F)\hookrightarrow \widetilde{SL}(2,F)$ is defined as follows:
\[
 s(g) = (\iota(g), m(\iota(g))^{-1}\beta_V(g))
\]
It is easily verified that this is a group homomorphism. One important thing to verify is that $m(\iota(g))^{-1}\beta_V(g)$ takes $\pm 1$ value. Further calculation shows:
\[
 s: \begin{bmatrix}a&b\\0&a^{-1}\end{bmatrix} \mapsto \left(\iota(\begin{bmatrix}a&b\\0&a^{-1}\end{bmatrix}), (a,\Delta)\right)
\]
\[
 s: \begin{bmatrix} 0&1\\ -1&0\end{bmatrix} \mapsto \left(\iota(\begin{bmatrix} 0&1\\ -1&0\end{bmatrix}), (-1,\Delta)\right)
\]
generate this embedding. 

It is an elementary consequence of the theory of the theta correspondence that the restriction $\widetilde{\omega}_{\psi}$ to $SO(2)\times SL(2,F)$ can be written as direct sum:
\[
 \sum_i \theta_i \boxtimes \pi_i
\]
where $\theta_i$ is a character of $SO(2)$, and $\pi_i$ is an irreducible representation of $SL(2,F)$.

Let's first determine the irreducible subspace on which $SO(2)$ acts with the character $\theta$, i.e., the $\theta$-isotypic:
\[
 \mathcal{S}(E)_{\theta} = \{ \varphi\in \mathcal{S}(E)| \varphi (\alpha\cdot \mu) = \theta(\mu)\varphi(\alpha),\ \forall \mu \in E^1, \ \forall \alpha\in E\}
\]
The way $SL(2,F)$ acts on $\mathcal{S}(E)_{\theta}$ is completely determined by the embedding $s$ and the Weil representation $\widetilde{\omega}_{\psi}$:
\[
 \pi_{\theta,\psi}\begin{bmatrix}a&b\\c&d\end{bmatrix}\cdot \varphi(\alpha) = \widetilde{\omega}_{\psi}\left(s(\begin{bmatrix}a&b\\c&d\end{bmatrix})\right)
\]
Formulas for the Schr\"odinger model of $\widetilde{\omega}_{\psi}$ is given in the Appendix.
The Theorem then follows by simple calculation.

\end{proof}

\begin{rmk}
 Notice that this correspondence is 2-to-1, because $\theta$ and $\theta^{\sigma}$ both corresponds to isomorphic $SL(2,F)$ representations. Fortunately, this won't get in the way of constructing the Langlands correspondence, since later we will be using the full information of $\Theta$ to construct a representation of $GL(2,F)$. $\Theta$ and $\Theta^{\sigma}$ will correspond to different representations of $GL(2,F)$.
\end{rmk}

\section{The Langlands Correspondence}

Let $(E/F,\Theta)$ be an admissible pair, we have constructed an irreducible representation $\pi_{\theta,\psi}$ of $SL(2,F)$, now we would like to inflate this representation to a representation of $GL(2,F)$.

\begin{defn}
 Let the central extension of $SL(2,F)$ be 
 \[
  GL(2,F)_{\circ} = SL(2,F) Z
 \]
where 
\[
 Z = \{ \begin{bmatrix} x&0 \\ 0&x \end{bmatrix} |x\in F^{\times}\}
\]
i.e., $GL(2,F)_{\circ} = \{ g\in GL(2,F)| \det(g) \in (F^{\times})^2\}$
\end{defn}

To extend $\pi_{\theta,\psi}$, we need to know a little more about the Langlands correspondence.
\begin{prop}
 If $\boldsymbol{\pi}$ is the Langlands correspondence in Definition~\ref{locallanglands}, then for $\rho\in \bg_2^0(F)$ and $\pi = \boldsymbol{\pi}(\rho)$, $\chi_{\pi} = \det \rho$. $\chi_{\pi}$ is the central character of $\pi$.
\end{prop}

\begin{lemma}
 \[
  \det \rho_{\Theta} = \kappa\otimes \Theta|_F
 \]
where $\kappa(x) = (x,\Delta)$, $\rho_{\Theta} = \text{Ind}_{\mathcal{W}_E}^{\mathcal{W}_F}\Theta$
\end{lemma}
Now we have enough information to extend the representation $\pi_{\theta,\psi}$ to $GL(2,F)_{\circ}$.
\begin{defn}
 \[
  \pi(\Theta,\psi)_{\circ} \begin{bmatrix}x&0\\0&x \end{bmatrix} \cdot \varphi(\alpha) = \kappa(x) \Theta(x) \varphi(\alpha)
 \]
is an irreducible representation of $GL(2,F)_{\circ}$ on $\mathcal{S}(E)_{\theta}$.
\end{defn}
Consider short exact sequence:
\[
 0\rightarrow GL(2,F)_{\circ}\hookrightarrow GL(2,F)\xrightarrow{\det}F^{\times}/(F^{\times})^2\rightarrow 0
\]
To go from $GL(2,F)_{\circ}$ to $GL(2,F)$, we consider elements in $GL(2,F)$ whose determinant is a norm in $F^{\times}$:
\begin{defn}
 \[
  GL(2,F)_N = \{ g\in GL(2,F)| \det(g)\in N(E^{\times})\}
 \]
\end{defn}
$GL(2,F)_N$ can be generated by the group $GL(2,F)_{\circ}$ and $N(E^{\times})\cong \{ \begin{bmatrix} \gamma & 0 \\ 0& 1\end{bmatrix} | \gamma\in N(E^{\times})\}$ Let's extend the representation to $\begin{bmatrix} \gamma & 0 \\ 0& 1\end{bmatrix}$. Let $\pi_{\Theta,\psi}$ denote the representation of $GL(2,F)_N$ extended from $\pi(\Theta,\psi)_{\circ}$. Automatically, $\pi_{\Theta,\psi}$ has to satisfy:
\begin{equation*}
\begin{split}
 \pi_{\Theta,\psi}\begin{bmatrix} \gamma &0\\0&1\end{bmatrix}\cdot \begin{bmatrix} \gamma &0\\0&1\end{bmatrix}\cdot \varphi(\alpha) &= \pi(\Theta,\psi)_{\circ} \begin{bmatrix} \gamma^2 &0\\0&1\end{bmatrix}\cdot \varphi(\alpha)\\
 &=\Theta(\zeta^2)\cdot |\gamma|\cdot \varphi(\sigma(\zeta^2)\alpha)
\end{split}
\end{equation*}
Here $\zeta\in E$ is defined by $N(\zeta) = \gamma$.
\begin{defn}
 \[
  \pi_{\Theta,\psi}\begin{bmatrix} \gamma &0\\0&1\end{bmatrix} \cdot \varphi(\alpha) = \Theta(\zeta)\cdot |\gamma|^{\frac{1}{2}}\cdot \varphi(\sigma(\zeta)\alpha), \ \ \forall \gamma\in N(E^{\times})
 \]
 \[
  \pi_{\Theta,\psi} \begin{bmatrix}x&0\\0&x \end{bmatrix} \cdot \varphi(\alpha) = (x,\Delta) \Theta(x) \varphi(\alpha), \ \ \forall x\in F^{\times}
 \]
 together with the action of $SL(2,F)$ via $\pi_{\theta,\psi}$ defines a representation of $GL(2,F)_N$.
\end{defn}
\begin{rmk}
 This extension is independent of the choice of $\zeta$ since if we choose another element with norm $\alpha$, it is going to be of the form $\zeta\mu$, $\mu\in E^1$. 
 \begin{equation*}
  \begin{split}
   \Theta(\zeta\mu)|\gamma|^{\frac{1}{2}}\varphi(\sigma(\zeta)\sigma(\mu)\alpha) &= \Theta(\zeta)\theta(\mu)|\gamma|^{\frac{1}{2}}\theta(\sigma(\mu))\varphi(\sigma(\zeta)\alpha) \\
   &= \Theta(\zeta)\theta(N(\mu))|\gamma|^{\frac{1}{2}}\varphi(\sigma(\zeta)\alpha) =\Theta(\zeta)|\gamma|^{\frac{1}{2}}\varphi(\sigma(\zeta)\alpha)
  \end{split}
 \end{equation*}
 \end{rmk}
Finally, we would like to obtain an irreducible representation of $GL(2,F)$ from $\pi_{\Theta,\psi}$. Since $GL(2,F)_N$ is an index 2 subgroup of $GL(2,F)$, we would like to induce $\pi_{\Theta,\psi}$. Note that the non-trivial coset of $GL(2,F)_N$ in $GL(2,F)$ is represented by matrix $\begin{bmatrix}\eta& 0 \\ 0&1\end{bmatrix}$, where $\eta\in F^{\times}\backslash N(E^{\times})$.
\begin{lemma}
 For $t\in F^{\times}$, $\pi_{\theta,\psi}^t\cong \pi_{\theta,\psi_t}$.
\end{lemma}
Here $\pi_{\theta,\psi}^t$ is defined as 
\[
 \pi_{\theta,\psi}^t(A) = \pi_{\theta,\psi}(TAT^{-1})
\]
where $T = \begin{bmatrix}t&0\\0&1\end{bmatrix}$ and $\psi_t(\alpha) = \psi(t\alpha)$. This lemma is easy to prove from the formulas.

\begin{lemma}
 \[
  \pi_{\theta,\psi} \text{ is } \begin{cases} \cong \pi_{\theta,\psi_{t}} &  t\in N(E^{\times}) \\ \ncong \pi_{\theta,\psi_{t}} & t\notin N(E^{\times}) \end{cases}
 \]
\end{lemma}
 The fact $\pi_{\theta,\psi}\cong \pi_{\theta,\psi_t}$ when $t\in N(E^{\times})$ is easy to prove from the formulas by writing down an explicit intertwining operator. While the second part can be observed from the difference in their character formulas. Such character formulas are presented in \cite{SS}.

\begin{prop}\label{repofglNtwistbyt}
\[
 \pi_{\Theta,\psi} \text{ is } \begin{cases} \cong \pi_{\Theta,\psi_{t}} &  t\in N(E^{\times}) \\ \ncong \pi_{\Theta,\psi_{t}} & t\notin N(E^{\times}) \end{cases}
 \]
\end{prop}
This is an elementary consequence of the previous lemmas. 

  \begin{defn}
   \[
    \pi_{\Theta} = \text{Ind}_{GL(2,F)_N}^{GL(2,F)} \pi_{\Theta,\psi}
   \]
  \end{defn}

\begin{thm}\label{independentofpsi}
 $\pi_{\Theta}$ is irreducible and independent of $\psi$.
\end{thm}
\begin{proof}
 $GL(2,F)_N$ is an index 2 subgroup in $GL(2,F)$. By Clifford theory, the induced representation $\pi_{\Theta}$ is irreducible if and only if $\pi_{\Theta,\psi}^{\eta}$ is not isomorphic to $\pi_{\Theta,\psi}$. $\eta$ is not in $N(E^{\times})$ by definition, therefore Proposition~\ref{repofglNtwistbyt} implies $\pi_{\Theta}$ is irreducible. 
 
 To prove the independence of $\psi$, we only need to consider Ind$_{GL(2,F)_N}^{GL(2,F)}\pi_{\Theta,\psi}$ and $\text{Ind}_{GL(2,F)_N}^{GL(2,F)} \pi_{\Theta,\psi_{\eta}}$. Note the representation Ind$_{GL(2,F)_N}^{GL(2,F)}\pi_{\Theta,\psi}$ restricted to $GL(2,F)_N$ decomposes into $\pi_{\Theta,\psi}\oplus \pi_{\Theta,\psi}^{\eta}$. By Frobenius reciprocity we have:
 \[
  \text{Ind}_{GL(2,F)_N}^{GL(2,F)}\pi_{\Theta,\psi}\cong \text{Ind}_{GL(2,F)_N}^{GL(2,F)} \pi_{\Theta,\psi_{\eta}}
 \]
hence the independence of $\psi$.
\end{proof}
Now we have paved most of our way towards the Langlands correspondence. We will state a partial correspondence theorem. 
\begin{defn}
 Let $\rho\in \bg_2^0(F)$; one say that $\rho$ is \emph{imprimitive} if there exists a separable quadratic extension $E/F$ and a character $\xi$ of $E^{\times}$ such that $\rho\cong \text{Ind}_{E/F} \xi$. Let $\bg_2^{im}(F)$ denote the set of imprimitive equivalence classes $\rho\in \bg_2^0(F)$.
\end{defn}

\begin{thm}\cite[40.1]{BH}
 \[
  \rho_{\Theta}\mapsto (E/F,\Theta)\mapsto \pi_{\Theta}
 \]
is a bijection $\bg_2^{im}(F)\leftrightarrow \ba_2^0(F)$ and it satisfies Equation~\ref{epsiloncondition} and Equation~\ref{Lcondition}.
\end{thm}
The proof of this theorem can be found in \cite{BH} section 39. One only need to make the observation that $\pi_{\Theta,\psi}\cong \xi_{\varkappa}(\Theta,\psi)$, the latter notation is from \cite{BH}. The intertwining operator is presented in \cite[39.2]{BH}.
\begin{rmk}
 When the residual characteristic of $F$ is odd, we have $\bg_2^{im}(F) = \bg_2^0(F)$. If $p = 2$, we have $\bg_2^{im}(F)\subsetneq \bg_2^0(F)$. In which case the theorem only provides a partial correspondence.
\end{rmk}
\section{Appendix}
In this section, $W$ is a symplectic vector space with skew-symmetric bilinear form $\lgl,\rgl$, $X$ and $Y$ are two transversal Lagrangian subspaces of $W$. $Sp(W)$ has Bruhat decomposition $Sp(W) = P\Omega P$ where $P$ is the stablizer of $Y$, $\Omega$ is the Weyl group.
\begin{thm}\cite[Theorem4.1]{RR}
 \[
  c(g_1,g_2) = \text{Weil index of }w\mapsto \psi(\frac{1}{2}\lgl w, w\cdot \rho \rgl )
 \]
where the isometry class of $\rho$ is given by the Leray invariant $q(Y,Yg_2^{-1},Yg_1)$
\end{thm}
\begin{lemma}\cite[Lemma5.1]{RR}
 There exists a unique map $f:Sp(W)\rightarrow F^{\times}/(F^{\times})^2$ such that the following holds:
 \begin{enumerate}[(i)]
  \item  $f(p_1gp_2) = f(p_1)f(g)f(g_2)\ \ \forall p_1,p_2\in P$
  \item $f(\tau_S) = 1$ for all subsets $S\subset \{1,2,\cdots,n\}$
  \item $f(p) = \det(p|_Y)(F^{\times})^2\ \ \forall p \in P$
 \end{enumerate}
Moreover, such a function is uniquely defined by 
\[
 f(p_1\tau_Sp_2) = \det(p_1p_2|_Y)(F^{\times})^2
\]
\end{lemma}
See \cite[P345]{RR} for the definition of $\tau_S$.

\begin{defn}
 Define the normalizing constant
 \[
  m(g) = \gamma_F(f(g),\frac{1}{2}\psi)^{-1}\left(\gamma_F(\frac{1}{2}\psi)\right)^{-j}
 \]
for $g\in \Omega_j P\tau_S P$ with $j = |S|$, $\gamma_F(\frac{1}{2}\psi)$ is the Weil index of $\alpha\mapsto \frac{1}{2}\psi(\alpha^2)$, $\gamma_F(a,\psi) = \gamma_F(\psi_a)/\gamma_F(\psi)$.
\end{defn}

\begin{thm}\cite[P364]{RR}
 If $\dim W = 2$, then 
 \[
  \tilde{c}(g_1g_2) = (f(g_1),f(g_2))(-f(g_1)f(g_2),f(g_1g_2))
 \]
\[
 f:\begin{bmatrix}a&b\\c&d \end{bmatrix} \mapsto \begin{cases} d(F^{\times})^2&c=0 \\ c(F^{\times})^2&c\neq 0 \end{cases}
\]
\end{thm}

\begin{thm}\label{Weilformulas}
Here are the explicit formulas for the Weil representation $\tilde{\omega}_{\psi}$ of $Mp(4)$:
\begin{equation*}
\begin{split}
& \tilde{\omega}_{\psi}\left(\begin{bmatrix}
A&0\\0&^tA^{-1}\end{bmatrix}, \epsilon\right)\cdot \varphi (\alpha) = \epsilon\cdot  \frac{\gamma(\frac{1}{2}\psi)}{\gamma(\det A\cdot \frac{1}{2}\psi)}\lvert \det A\rvert^{\frac{1}{2}}\varphi(\alpha\cdot  A) \\
& \tilde{\omega}_{\psi}\left(\begin{bmatrix}
I&B\\0&I\end{bmatrix}, \epsilon\right)\cdot \varphi(x) = \epsilon\cdot \psi\left(\frac{1}{2}\langle x,x\cdot B\rangle\right)\varphi(x) \\
& \tilde{\omega}_{\psi}\left( \tau_{\{1\}},\epsilon \right)\cdot \varphi(x) = \epsilon\cdot \gamma(\frac{1}{2}\psi)^{-1} \int_{Y/Y_2} \psi(-b_1\alpha_1)\varphi(-b_1x_1+\alpha_2x_2)d\mu_g(b_1) \\
& \tilde{\omega}_{\psi}\left( \tau_{\{2\}},\epsilon \right)\cdot \varphi(x) = \epsilon\cdot \gamma(\frac{1}{2}\psi)^{-1}\int_{Y/Y_1}\psi(-b_2\alpha_2)\varphi(\alpha_1x_1-b_2x_2)d\mu_g(b_2) \\
& \tilde{\omega}_{\psi}\left( \tau_{\{1,2\}},\epsilon \right)\cdot \varphi(x) = \epsilon\cdot \gamma(\frac{1}{2}\psi)^{-2} \int_Y \psi(\langle \alpha,y\rangle)\varphi(y)d\mu_g(y) \\
& \mu_p\{0\} = \lvert\det p|_Y\rvert^{-1/2}
\end{split}
\end{equation*}

$Y_i = span_{\mathbb{C}}\{y_i\}$, $y\in Y$ is represented by coordinates $(b_1,b_2)$ with basis $\{y_1,y_2\}$. $\alpha\in E$ is represented by coordinates $(\alpha_1,\alpha_2)$ with basis $\{ 1, -\frac{1}{\Delta}\delta\}$
\end{thm}

\bibliography{references}{}
\bibliographystyle{plain}
\nocite{BH, Bump, RR, Prasad, MVW, Howe, Adams, Kudla, SS, KudlaSplitting}
\end{document}